\documentclass{amsart}
\usepackage{amssymb,hyperref}
\usepackage{latexsym}
\usepackage{amssymb}
\usepackage{amsfonts}
\usepackage{amsmath}
\usepackage{amsthm}
\usepackage{graphics}
\usepackage[pdftex]{graphicx}
\usepackage{tikz}

\newtheorem{lemma}{Lemma}
\newtheorem{theorem}[lemma]{Theorem}
\newtheorem{proposition}[lemma]{Proposition}
\newtheorem{corollary}[lemma]{Corollary}
\newtheorem{example}{Example}
\numberwithin{equation}{section}

\newtheorem{remark}[lemma]{Remark}

\newtheorem{definition}{Definition}
\newtheorem{notation}{Notation}
\newtheorem{case}{Case}

\def\Sm0#1{{{\rm GL}}(1,#1)}

\title[Bipartite divisor graph for character degrees]
{Bipartite divisor graph for the set of irreducible character degrees}

\author{Roghayeh Hafezieh}

\address{Roghayeh Hafezieh, Department of Mathematics, Gebze Technical University, Gebze, Turkey}
\email{roghayeh@gyte.edu.tr}

\begin{document}

\maketitle

\begin{abstract}
Let $G$ be a finite group. We consider the set of the irreducible complex characters of $G$, namely $Irr(G)$, and the related degree set $cd(G)=\{\chi(1) : \chi\in Irr(G)\}$. Let $\rho(G)$ be the set of all primes which divide some character degree of $G$. In this paper we introduce the bipartite divisor graph for $cd(G)$ as an undirected bipartite graph with vertex set $\rho(G)\cup (cd(G)\setminus\{1\})$, such that an element $p$ of $\rho(G)$ is adjacent to an element $m$ of $cd(G)\setminus\{1\}$ if and only if $p$ divides $m$. We denote this graph simply by $B(G)$. Then by means of combinatorial properties of this graph, we discuss the structure of the group $G$. In particular, we consider the cases where $B(G)$ is a path or a cycle.
\end{abstract}

Classification: Primary 05C25; secondary 05C75.


\section{Introduction}
\label{sec:introd}
Let $G$ be a finite group. It is well known that the set of irreducible characters of $G$, denoted by $Irr(G)$, can be used to obtain information about the structure of the group $G$. The value of each character at the identity is the degree of the character and by $cd(G)$ we mean the set of all irreducible character degrees of $G$. Let $\rho(G)$ be the set of all primes which divide some character degree of $G$. When studying problems on character degrees, it is useful to attach the following graphs which have been widely studied to the sets $\rho(G)$ and $cd(G)\setminus\{1\}$, respectively.
\begin{itemize}
\item[(i)] Prime degree graph, namely $\Delta(G)$, which is an undirected graph whose set of vertices is $\rho(G)$; there is an edge between two different vertices $p$ and $q$ if $pq$ divides some degree in $cd(G)$.
\item[(ii)] Common divisor degree graph, namely $\Gamma(G)$, which is an undirected graph whose set of vertices is $cd(G)\setminus\{1\}$; there is an edge between two different vertices $m$ and $n$ if $(m,n)\neq1$.
\end{itemize}
The notion of the bipartite divisor graph was first introduced by Iranmanesh and Praeger in ~\cite{IP} for a finite set of positive integers. As an application of this graph in group theory, in ~\cite{BDIP}, the writers considered this graph for the set of conjugacy class sizes of a finite group and studied various properties of it. In particular they proved that the diameter of this graph is at most six, and classified those groups for which the bipartite divisor graphs of conjugacy class sizes have diameter exactly $6$.
Moreover, they showed that if the graph is acyclic, then the diameter is at most five and they classified the groups for which the graph is a path of length five. Similarly,  Taeri in ~\cite{T}
considered the case that the bipartite divisor graph of the set of conjugacy class sizes is a cycle and (by using the structure of $F$-groups and the classification of finite simple groups) proved that for a finite nonabelian group $G$, the bipartite divisor graph of the conjugacy class sizes is a cycle if and only if it is a cycle of length $6$, and for an abelian group $A$ and $q\in\{4,8\}$, $G\simeq A\times SL_2(q)$. Inspired by these papers, in this work we consider the bipartite divisor graph for the set of irreducible character degrees of a finite group and define it as follows:
\begin{definition}
Let $G$ be a finite group. The bipartite divisor graph for the set of irreducible character degrees of $G$, is an undirected bipartite graph with vertex set $\rho(G)\cup (cd(G)\setminus\{1\})$, such that an element $p$ of $\rho(G)$ is adjacent to an element $m$ of $cd(G)\setminus\{1\}$ if and only if $p$ divides $m$.
\end{definition}
Since classifying groups whose associated graphs have special graphical shapes is an important topic in this area, in this paper, we will discuss the cases where $B(G)$ is a path or a cycle for a group $G$.

In the second section, we suppose that $G$ is a solvable group. After finding the best upper bound for diameter of $B(G)$, we consider the case where $B(G)$ is a path of length $n$. We prove $n\leq 6$ and in Theorem ~\ref{thm:2}, which is the main theorem of this section, we give some group theoretical properties of such a group.

In the third section, we consider the case that $G$ is nonsolvable and $B(G)$ is a union of paths where by union of paths we mean that each connected components of $B(G)$ is a path, (so $B(G)$ is a path if there exists only one path in this union). Theorem ~\ref{thm:99} is the main theorem of this section.

Finally in section four, we consider the case where $B(G)$ is a cycle. We prove that $B(G)$ is a cycle if and only if $G$ is solvable and $B(G)$ is a cycle of length four or six. By using these properties, we find a special subgroup of $G$ which explains the structure of the irreducible character degrees of $G$. Theorem ~\ref{thm:98} is the main theorem of this section.

\begin{notation}
 For positive integers $m$ and $n$, we denote the greatest common divisor of $m$ and $n$ by
$(m,n)$; the number of connected components of a graph $\mathcal{G}$ by $n(\mathcal{G})$; the diameter of a graph $\mathcal{G}$ by $diam(\mathcal{G})$ (where by the diameter we mean the
maximum distance between vertices in the same connected component of the graph). If $\alpha$ is a vertex of the graph $\mathcal{G}$, then $deg_{\mathcal{G}}(\alpha)$ is the number of vertices adjacent to $\alpha$ in $\mathcal{G}$. If the graph is well-understood, then we show it by $deg(\alpha)$.
By length of a path or a cycle, we mean the number of edges in the path or in the cycle. Also, by $P_{n}$ and $C_{n}$ we mean a path of length $n$ and
a cycle of length $n$, respectively. Let
$G$ be a finite solvable group. As usual, we write $dl(G)$ and $h(G)$ to denote the derived length and Fitting height of $G$, respectively. Other notation
throughout the paper is standard.
\end{notation}
\section{ Solvable groups whose bipartite divisor graphs are paths }
\label{sec:Path}
We begin by giving the best upper bound for $diam(B(G))$.
\begin{theorem}
For a finite solvable group $G$, $diam(B(G))\leq 7$ and this bound is best possible.
\end{theorem}
\begin{proof}
Since $G$ is solvable, by ~\cite[Corollary 4.2, Theorem 7.2]{L} we know that $diam(\Delta(G))$ and $diam(\Gamma(G))$ are both less than or equal to three. Now by ~\cite[Lemma 1]{IP} we have one of the following cases:
\begin{itemize}
\item[(i)] $diam(B(G))=2 max\{diam(\Delta(G)),diam(\Gamma(G))\}\leq 2\times 3 =6$ or
\item[(ii)] $diam(B(G))=2 diam(\Delta(G))+1= 2 diam(\Gamma(G))+1\leq (2\times 3)+1 =7$.
\end{itemize}
So in general, $diam(B(G))\leq 7$. Now let $G$ be the group as in ~\cite{1}, then

$cd(G)=\{1, 3, 5, 3\times 5, 7\times 31\times 151, 2^{7}\times 7\times 31\times 151, 2^{12}\times 31\times 151, 2^{12}\times 3\times 31\times 151,$
$$ 2^{12}\times 7\times 31\times 151, 2^{13}\times 7\times 31\times 151, 2^{15}\times 3\times 31\times 151\}$$
It is easy to see that $diam(\Delta(G))=diam(\Gamma(G))=3$ and $diam(B(G))=7$.
\end{proof}
\begin{proposition}\label{prop:1}
Let $G$ be a finite solvable group. Assume that $B(G)$ is a path of length $n$. Then $n\leq 6$ and $dl(G)\leq 5$. In particular, if $B(G)$ equals $P_5$ or $P_6$, then $h(G)$ is less than or equal $3$ or $4$, respectively. (In general, we know that for a solvable group $h(G)\leq |cd(G)|$.)
\end{proposition}
\begin{proof}
Since $B(G)=P_{n}$, we can see that both $\Delta(G)$ and $\Gamma(G)$ are paths, (see ~\cite[Theorem 3]{IP}). Since $G$ is solvable, by ~\cite[Theorem 4.1]{L}, we conclude that $\Delta(G)=P_m$ where $m\leq 3$. Furthermore ~\cite[Theorem 4.5]{L} implies that a path of length three cannot be the prime degree graph of a solvable group, so $m\leq 2$. On the other hand we know that $|diam(\Delta(G))-diam(\Gamma(G))|\leq 1$, so $n\leq 6$ which implies that $|cd(G)|\leq 5$. Now by ~\cite{7}, we have $dl(G)\leq 5$. In the case that $B(G)$ equals $P_5$ or $P_6$, since $|cd(G)|\geq 4$, ~\cite{2} verifies that $h(G)\leq |cd(G)|-1$. Thus $h(G)\leq 3$ if $B(G)=P_5$ and $h(G)\leq 4$ if $B(G)=P_6$.
\end{proof}
\begin{theorem}\label{thm:2}
Let $G$ be a finite solvable group. Assume that $B(G)$ is a path of length $n$. Then we have the following cases:
\begin{itemize}
\item[(i)]$G\simeq P\times A$, where $P$ is a $p$-group for some prime number $p$ and $A$ is an abelian group.
\item[(ii)] There exist normal subgroups $N$ and $K$ of $G$ and a prime number $p$ with the following properties:
  \begin{itemize}
    \item[(a)] $\frac{G}{N}$ is abelian.
    \item[(b)] $\pi(G/K)\subseteq\rho(G)$.
    \item[(c)] Either $p$ divides all the nontrivial irreducible character degrees of $N$, or there exists a unique nontrivial $\psi(1)\in cd(N)$ such that $[G:N]\psi(1)\in cd(G)$.
  \end{itemize}
\item[(iii)] $cd(G)=\{1,p^{\alpha},q^{\beta},p^{\alpha}q^{\beta}\}$, where $p$ and $q$ are distinct primes.
\item[(iv)] There exists a prime $s$ such that $G$ has a normal $s$-complement.
\item[(v)] $G$ has an abelian normal subgroup $N$ such that $[G : N] = m\in cd(G)^{*}$.
\end{itemize}
\end{theorem}
\begin{proof}
Since $B(G)$ is a path of length $n$, Proposition~\ref{prop:1} implies that $n\leq 6$.

First suppose that $n\geq 4$. This implies that nonlinear irreducible character degrees of the solvable group $G$ are not all equal.
We claim that there exists a normal subgroup $K>1$ of $G$ such that $\frac{G}{K}$ is nonabelian.
If $G'$ is not a minimal normal subgroup of $G$, then $G$ has a nontrivial normal subgroup $N$ such that $\frac{G}{N}$ is nonabelian. Otherwise, if $G'$ is a minimal normal subgroup of $G$, then it cannot be unique since nonlinear irreducible character degrees of the solvable group $G$ are not all equal, ~\cite[Lemma 12.3]{5}. So we can see that $G$ has a nontrivial normal subgroup $N$ such that $\frac{G}{N}$ is nonabelian. Let $K$ be maximal with respect to the property that $\frac{G}{K}$ is nonabelian. It is clear that $(\frac{G}{K})'$ is the unique minimal normal subgroup of $\frac{G}{K}$. Thus $\frac{G}{K}$ satisfies the hypothesis of ~\cite[Lemma 12.3]{5}. So all nonlinear irreducible characters of $\frac{G}{K}$ have equal degree $f$ and we have the following cases:

\begin{case}
 $B(G)=P_6$.
  \end{case}
  In this case $|\rho(G)|$ is either $3$ or $4$. By ~\cite[Theorem 4.5]{L}, $\Delta(G)$ cannot be a path of length $3$, so the case $|\rho(G)|=4$ is impossible. We may assume that $B(G): m-p-n-q-l-r-k$, where $p$, $q$, and $r$ are distinct prime numbers. We claim that $\frac{G}{K}$ is not an $s$-group for a prime $s$. If not, then $s\in\{p,q,r\}$. As $cd(G)$ does not contain a degree which is a power of $q$, it is clear that $s \neq q$. Thus either $s=p$ and $cd(\frac{G}{K})=\{1,f=m=p^{\alpha}\}$ for a positive integer $\alpha$ or $s=r$ and $cd(\frac{G}{K})=\{1,f=k=r^{\beta}\}$ for a positive integer $\beta$. Since the roles of $p$ and $r$ are the same, without loss of generality, we may assume that $s=p$ and there exists $\theta\in Irr(\frac{G}{K})$ such that $\theta(1)=p^{\alpha}=m$. Let $\chi\in Irr(G)$ with $\chi(1)=k$. Since $p$ does not divide $\chi(1)$, we have $\chi_{K}\in Irr(K)$. Now by Gallagher's Theorem ~\cite{5}, we conclude that $\chi\theta\in Irr(G)$, which is distinct from $\theta$. Thus $\chi(1)\theta(1)=mk\in cd(G)^{*}$, so either $mk=l$ or $mk=n$. Since $(n,k)=1$ and $(l,m)=1$, none of these cases are possible, so $\frac{G}{K}$ is not an $s$-group. Now ~\cite[Lemma 12.3]{5} implies that $\frac{G}{K}$ is a Frobenius group with abelian Frobenius complement of order $f$ and Frobenius kernel $\frac{N}{K}=(\frac{G}{K})'$ which is an elementary abelian $s$-group for some prime $s$. Thus $[G/K:N/K] = [G:N] = f\in cd(\frac{G}{K})\subseteq cd(G) =\{1,m,n,k,l\}$. It is clear that $N$ is not abelian. First suppose that $f=m$ and let $\chi$ be a nonlinear irreducible character of $G$ with $\chi(1)=k$. If $s\notin \rho(G)$, then $(\chi(1),[G:K])=1$. This verifies that $\chi_{K}\in Irr(K)$ and by Gallagher's Theorem ~\cite{5} we have $\chi(1)f=km\in cd(G)$ which is impossible. Therefore $\pi(G/K)\subseteq \rho(G)$. Let $\psi$ be a nonlinear irreducible character of $N$. If $s$ does not divide $\psi(1)$, then by ~\cite[Theorem 12.4]{5} we conclude that $[G:N]\psi(1)\in cd(G)$. Now for any $1\neq \zeta(1)\in cd(N)$ which is different from $\psi(1)$, $[G:N]\psi(1)\notin cd(G)$. This implies that $s | \zeta(1)$ and $\psi(1)$ is a unique element in $cd(N)^{*}$ with respect to the property that $[G:N]\psi(1)\in cd(G)$. If $w$ is a nontrivial element of $cd(\frac{G}{N})$, then by Gallagher's Theorem ~\cite{5} we can see that $kw\in cd(G)$ which is impossible. Thus $\frac{G}{N}$ is abelian and case $(ii)$ holds. The case $f=k$ is similar.

 Now assume $f=n$. Suppose that $\psi\in Irr(N)$ with $\psi(1)\neq 1$. Since $(l,m)=1$, $(m,k)=1$ and $(n,k)=1$, we can see that $[G:N]\psi(1)\notin cd(G)$. Now by
~\cite[Theorem 12.4]{5}, we conclude that $s | \psi(1)$. Since $\psi\in Irr(N)$ was arbitrary, we deduce that $s$ divides all nontrivial irreducible character degrees of $N$. Furthermore, this implies that $\pi(\frac{G}{K})\subseteq \rho(G)$. As $\frac{G}{K}$ is Frobenius, we have $(s,[G:N])=1$. Since $f=n$, we have $s=r$. 
If $w$ is a nontrivial element of $cd(\frac{G}{N})$, then Gallagher's Theorem ~\cite{5} implies that $kw\in cd(G)$, so $kw=l$. This forces $w$ to be a power of $q$ which is impossible since $cd(\frac{G}{N})\subseteq cd(G)$. Thus $\frac{G}{N}$ is abelian and case $(ii)$ holds. The case $f=l$ is similar.
\begin{case}
 $B(G)=P_5$.
  \end{case}
  We may assume that $B(G): p-m-q-l-r-n$, where $p$, $q$, and $r$ are distinct prime numbers. We claim that $\frac{G}{K}$ is not an $s$-group. If $\frac{G}{K}$ is an $s$-group for some prime $s$, then there exists $\eta\in Irr(\frac{G}{K})$ such that $\eta(1)=s^{\alpha}\in cd(G)^{*}=\{m,n,l\}$. According to the form of $B(G)$, we can see that $s=r$ and $\eta(1)=r^{\alpha}=n$, (since $n$ has only one prime divisor). Let $\chi\in Irr(G)$ with $\chi(1)=m$. Since $(m,n)=1$, $r$ does not divide $m$. Thus $\chi_{K}\in Irr(K)$.
Now by Gallagher's Theorem ~\cite{5}, we conclude that $\chi(1)\eta(1)=mn\in cd(G)^{*}$, so $l=mn$. But in this case $l$ has three distinct prime divisors, which contradicts the form of $B(G)$. Hence $\frac{G}{K}$ is not an $s$-group. Now ~\cite[Lemma 12.3]{5} implies that $\frac{G}{K}$ is a Frobenius group with abelian Frobenius complement of order $f$ and Frobenius kernel $\frac{N}{K}=(\frac{G}{K})'$ which is an elementary abelian $s$-group for some prime $s$. So $[G/K:N/K] = [G:N] = f\in cd(\frac{G}{K})\subseteq cd(G) =\{1,m,n,l\}$.
Similar to the previous case, we can see that for $f=n$, case $(ii)$ occurs.
So suppose that $f\neq n$. If $\psi\in Irr(N)$ with $\psi(1)\neq 1$, then $[G:N]\psi(1)\notin cd(G)$. Now ~\cite[Theorem 12.4]{5} implies that $s | \psi(1)$. So either $f=l$ and $s=p$ or $f=m$ and $s=r$. Let $f=l$, $s=p$ and $\zeta$ be an irreducible constituent of $\psi^{G}$. Then $\zeta(1)=m$. Since $([G:N],\zeta(1))=1$, we have $\zeta_{N}\in Irr(N)$. Thus $\zeta_{N}=\psi$ and $\psi(1)=m$. Since any character degree of $\frac{G}{N}$ must divide the order of the group, $\frac{G}{N}$ is abelian and case $(ii)$ occurs. The case $f=m$ and $s=r$ is similar.

\begin{case}
 $B(G)=P_4$.
 \end{case}
 First suppose that $B(G): p-m-q-n-r$, where $p$, $q$, and $r$ are distinct prime numbers. Since $q$ divides every nonlinear character degree, $G$ has a normal $q$-complement, (~\cite[Corollary 12.2]{5}). Thus case $(iv)$ occurs with $s=q$.

 Now suppose that $cd(G)=\{1,m,n,l\}$. We may assume that $B(G): m-q-l-p-n$, where $p$ and $q$ are distinct prime numbers.
Suppose $\frac{G}{K}$ is a Frobenius group with abelian Frobenius complement of order $f$ and Frobenius kernel $\frac{N}{K}=(\frac{G}{K})'$ which is an elementary abelian $s$-group for some prime $s$. If $f=m$ and $s\notin \rho(G)$, then by Gallagher's Theorem ~\cite{5} we can conclude that $cd(G)=\{1,p^{\alpha},q^{\beta},p^{\alpha}q^{\beta}\}$, so case $(iii)$ occurs. If $f=m$ and $s\in \rho(G)$, then similar to the previous cases we can see that $s$ divides each nonlinear irreducible character degree of $N$ and $\frac{G}{N}$ is abelian. Thus case $(ii)$ occurs. The case $f=n$ is similar.
Let $f=l$. Since $(s,f)=1$ and $\frac{N}{K}$ is nontrivial, we have $s\notin \rho(G)$. On the other hand, for $\psi\in Irr(N)$ with $\psi(1)\neq 1$, $[G:N]\psi(1)\notin cd(G)$. ~\cite[Theorem 12.4]{5} implies that $s | \psi(1)$ and $s\in \rho(G)$ which is a contradiction. So $N$ is abelian. Since $N$ is normal and abelian, $\chi(1)$ divides $[G:N]=f$, for each $\chi\in Irr(G)$, which contradicts the form of $B(G)$. So $f\neq l$.

Now suppose $\frac{G}{K}$ is an $s$-group for a prime $s$. It is obvious that $f\neq l$. Without loss of generality we may assume that there exists $\eta\in Irr(\frac{G}{K})$ such that $\eta(1)=p^{\alpha}=n$. Let $\chi\in Irr(G)$ with $\chi(1)=m$. As $p$ does not divide $m$, we have $\chi_{K}\in Irr(K)$. Now by Gallagher's Theorem ~\cite{5}, we conclude that $\chi(1)\eta(1)=mn\in cd(G)^{*}$, so $l=mn$. Thus $cd(G)=\{1,p^{\alpha},q^{\beta},p^{\alpha}q^{\beta}\}$ and case $(iii)$ holds. It should be mentioned that such a group exists. We can see that for $G=S_3\rtimes A_4$ which is a solvable group we have $cd(G)=\{1,2,3,6\}$.

Now suppose that $B(G)$ is a path of length $n$ where $n\leq 3$. We have the following cases:
\begin{case}
$B(G)=P_3$.
 \end{case}
 Assume $B(G): p-m-q-n$, where $p$ and $q$ are distinct prime numbers. Since $q$ divides every nonlinear character degree of $G$, we conclude that $G$ has a normal $q$-complement, (~\cite[Corollary 12.2]{5}). Thus case $(iv)$ occurs with $s=q$.
 \begin{case}
 $B(G)=P_2$.
  \end{case}
  Assume first that $cd(G)=\{1,m\}$ where $m$ is not a prime power. Then $m=p^{a}q^{b}$ for some primes $p\neq q$ and integers $a,b\geq 1$. Now by ~\cite[Theorem 12.5]{5}, $G$ has an abelian normal subgroup $N$ such that $[G:N]=m$, so case $(v)$ holds.

If $cd(G)=\{1,m,n\}$, then both $1<m<n$ are powers of a prime $s$ and so $G$ has a normal $s$-complement and case $(iv)$ holds.
\begin{case}
 $B(G)=P_1$.
  \end{case}
  Thus we have $cd(G)=\{1,p^{a}\}$ for some prime $p$ and integer $a\geq 1$. Now ~\cite[Theorem 12.5]{5} implies that either $G\simeq P\times A$, where $P$ is a $p$-group, $A$ is abelian and case $(i)$ holds, or $G$ has an abelian normal subgroup of index $p^{a}$ and case $(v)$ occurs with $m=p^{a}$.
\end{proof}
According to the proof of Theorem ~\ref{thm:2}, we have the following remarks.
\begin{remark}
Suppose that $G$ is a finite group and $q\in \rho(G)$ is an end point (a vertex of degree one) in $B(G)$. So $G$ is a finite group with exactly one irreducible character degree $m$ which is divisible by $q$. Let $Q\in Syl_{q}(G)$, then either $Q\lhd G$ or $U:= O_{q}(G)< Q$. In the second case, by ~\cite[Theorem A]{LMNH}, we have the following properties:
\begin{itemize}
\item[(i)] $U$ is abelian.
\item[(ii)] $\frac{Q}{U}$ is cyclic.
\item[(iii)] $|\frac{Q}{U}|=m_{q}$, which is the $q$-part of $m$.
\item[(iv)] $\frac{Q}{U}$ is a $TI$-set in $\frac{G}{U}$.
\end{itemize}
Now suppose that $G$ is a finite group which is not nilpotent and $B(G)=P_n$ is a path of length $n$. It is clear that $B(G)$ has at least one prime as an end point if and only if $n=1$, $n=2$ with $cd(G)=\{1,m\}$, $n=3$, $n=4$ with $cd(G)=\{1,m,n\}$ or $n=5$. Since $G$ is not nilpotent, there exists a prime $q\in\rho(G)$ such that $O_{q}(G)$ has the above properties.
\end{remark}
\begin{remark}
Suppose $G$ is a finite group such that $B(G)$ is $P_{6}$, $P_{5}$ or $P_{4}$ with $cd(G)=\{1,m,n\}$. There exist two distinct primes $p$ and $r$ such that they are not neighbors in the graph $\Delta(G)$. Now ~\cite[Theorem 5.1]{L} verifies that $l_{p}(G)\leq 2$ and $l_{r}(G)\leq 2$.
\end{remark}
\section{ Nonsolvable groups whose bipartite divisor graphs are union of paths}
\label{sec:nonsolvable}
Let $G$ be a finite nonsolvable group. By ~\cite[Theorem 6,4]{L}, we know that $\Delta(G)$ has at most three connected components. Since ~\cite{IP} implies that $n(B(G))=n(\Gamma(G))=n(\Delta(G))$, we conclude that $n(B(G))\leq 3$. In the rest of this section, we consider the case where each connected component of $B(G)$ is a path and we start by looking at simple groups.
\begin{lemma}
For a nonabelian simple group $S$, $B(S)$ is disconnected and all the connected components are paths if and only if $S$ is isomorphic to one of the following groups:
\begin{itemize}
\item[(i)] $PSL(2,2^{n})$ where $|\pi(2^{n}\pm 1)|\leq 2$;
\item[(ii)] $PSL(2,p^{n})$ where $p$ is an odd prime and $|\pi(p^{n}\pm 1)|\leq 2$.
\end{itemize}
\end{lemma}
\begin{proof}
By ~\cite{IP} we know that $n(B(S))=n(\Gamma(S))=n(\Delta(S))$. Since all connected components of $B(S)$ are paths, so are the connected components of $\Delta(S)$. This implies that $\Delta(S)$ has no triangles. Thus by ~\cite[Lemma 3.1]{H}, one of the following cases holds:
\begin{itemize}
\item[(i)] $S\simeq PSL(2,2^{n})$ where $|\pi(2^{n}\pm 1)|\leq 2$ and so $|\pi(S)|\leq 5$;
\item[(ii)] $S\simeq PSL(2,p^{n})$ where $p$ is an odd prime and $|\pi(p^{n}\pm 1)|\leq 2$ and so $|\pi(S)|\leq 4$.
\end{itemize}
Since $cd(PSL(2,2^{n}))=\{1,2^{n},2^{n}+1,2^{n}-1\}$, by ~\cite[Theorem 6.4]{L}, we conclude that $B(S)$ has three connected components while $B(S)$ has two connected components in the case $(ii)$. Also it is clear that in both cases all the connected components are paths.
\end{proof}
\begin{lemma}\label{lem:22}
Let $G$ be a finite group. Assume that $B(G)$ is a union of paths. If $|\rho(G)|=5$, then $G\simeq PSL(2,2^{n})\times A$, where $A$ is an abelian group and $|\pi(2^{n}\pm 1)|=2$.
\end{lemma}
\begin{proof}
Since each connected component of $B(G)$ is a path, we can see that $\Delta(G)$ is triangle-free. We claim that $B(G)$ is disconnected. If $B(G)$ is connected, then $\Delta(G)$ is a path. Now ~\cite[Theorem B]{H} implies that $G\simeq H\times K$, where $H$ is isomorphic with $A_{5}$ or $PSL(2,8)$ and $K$ is a solvable group. This contradicts that $\Delta(G)$ is a path. Thus $n(\Delta(G))=n(B(G))>1$. By ~\cite[Theorem B]{H}, we have $G\simeq PSL(2,2^{n})\times A$, where $A$ is an abelian group and $|\pi(2^{n}\pm 1)|=2$. So $B(G)$ is a graph with three connected components, one of them is a path of length one and the other two components are paths of length two.
\end{proof}
\begin{theorem}~\label{thm:99}
Let $G$ be a finite nonsolvable group. Assume that $B(G)$ is a union of paths. Then $B(G)$ is disconnected and we have the following cases:
\begin{itemize}
\item[(1)] If $n(B(G))=2$, then let $\mathcal{C}_{1}$ and $\mathcal{C}_{2}$ be the connected components of $B(G)$.
We have $|\rho(G)|\in\{3,4\}$, $\mathcal{C}_{1}$ is a path of length one and $\mathcal{C}_{2}$ is isomorphic with $P_{n}$, where $n\in\{|\rho(G)|, |\rho(G)|+1\}$.
\item[(2)] If  $n(B(G))=3$, then $G\simeq PSL(2,2^{n})\times A$, where $A$ is an abelian group and $n\geq 2$.
\end{itemize}
\end{theorem}
\begin{proof}
Since $B(G)$ is a union of paths, we deduce that $\Delta(G)$ is triangle-free. Now ~\cite[Theorem A]{H} implies that $|\rho(G)|\leq 5$. In Lemma~\ref{lem:22} we saw that if $|\rho(G)|=5$, then $G\simeq PSL(2,2^{n})\times A$, where $A$ is an abelian group and $|\pi(2^{n}\pm 1)|=2$. This implies that $n(B(G))=3$. On the other hand, since $G$ is a nonsolvable group, by ~\cite[Theorem 6.4]{L} we conclude that $n(B(G))=3$ if and only if $G\simeq PSL(2,2^{n})\times A$, where $A$ is an abelian group and $n\geq 2$. So we may assume that $|\rho(G)|\leq 4$ and $n(B(G))\leq 2$. It is obvious that $|\rho(G)|> 1$. Since $G$ is nonsolvable, it must have a nonabelian chief factor say $M/K\simeq S^{k}$, where $S$ is a nonabelian simple group and $k$ is a positive integer. Now it follows from $It\hat{o}-Michler$ and $Burnside's$ $p^{a}q^{b}$ theorems that $|\rho(S)|\geq 3$ and thus $|\rho(G)|\geq |\rho(M/K)|=|\rho(S)|\geq 3$.
Now we have the following cases:
\begin{itemize}
\item[(i)] $|\rho(G)|=3$. First suppose that $n(B(G))=2$. Thus $n(\Gamma(G))=2$. Now ~\cite[Theorem 7.1]{L} implies that one of the connected components of $\Gamma(G)$ is an isolated vertex and the other one has diameter at most two. This isolated vertex generates a connected component in $B(G)$ which is a path of length one. We denote this component by $\mathcal{C}_{1}$. Let $\mathcal{C}_{2}$ be the other component of $B(G)$. Since $G$ is nonsolvable, we can easily see that $\mathcal{C}_{2}$ is either a path of length $3$ or $4$. Now suppose that $B(G)$ is connected. Since $G$ is nonsolvable, we have $|cd(G)|>3$. Now by ~\cite[Corollary B]{MM}, we can conclude that $|cd(G)|\geq 5$. This implies that $\Gamma(G)$ is a path of length $3$ which is impossible by ~\cite[Theorem A]{HQ}. Note that the case $\Gamma(G)=P_{4}$ is not possible as $|\rho(G)|=3$.
\item[(ii)] $|\rho(G)|=4$. If $B(G)$ is connected, then $\Delta(G)$ is a path of length three which is impossible by ~\cite[Theorem B]{LW}, so we may assume that $n(B(G))=2$. Similar to the previous case, we can see that $\mathcal{C}_{1}$ is a path of length one and $\mathcal{C}_{2}$ is either a path of length $4$ or $5$.
\end{itemize}
\end{proof}
\begin{example}
Let $G_{1}=M_{10}$ and $G_{2}=PSL(2,25)$. Since $cd(M_{10})=\{1,9,10,16\}$ and $cd(PSL(2,25))=\{1,13,24,25,26\}$, it is easy to see that $B(G_{i})$ has two connected components $\mathcal{C}_{i,1}$ and $\mathcal{C}_{i,2}$, for $i\in\{1,2\}$. $\mathcal{C}_{i,1}$ is a path of length one for each $i$, $\mathcal{C}_{1,2}$ is a path of length $3$ and $\mathcal{C}_{2,2}$ is a path of length $5$.
\end{example}
\section{ Groups whose bipartite divisor graphs are cycles}
\label{sec:Cycle}
\begin{lemma}\label{rem:01}
Let $G$ be a finite group whose $B(G)$ is a cycle of length $n\geq 6$. Then both $\Delta(G)$ and $\Gamma(G)$ are cycles.
\end{lemma}
\begin{proof}
Suppose that $B(G)=C_{n}$ and $n\geq 6$. Let $\Phi\in\{\Delta(G),\Gamma(G)\}$. By graph theory we know that $\Phi$ is a cycle if and only if it is a connected graph such that every vertex in $\Phi$ has degree two. Since $n(B(G))=1$, ~\cite[Lemma 3.1]{IP} implies that $\Phi$ is a connected graph. Let $\alpha$ be a vertex of $\Phi$. It is clear that $deg_{B(G)}(\alpha)=2$. Since $n\geq 6$, one can see that $deg_{\Phi}(\alpha)=2$. Thus $\Phi$ is a cycle.
\end{proof}
\begin{theorem}\label{thm: 1}
Let $G$ be a finite group whose $B(G)$ is a cycle of length $n$. Then $n\in\{4,6\}$.
\end{theorem}
\begin{proof}
Since $B(G)$ is a cycle of length $n$, it is clear that $n\geq 4$. Furthermore, by ~\cite[Theorem 3]{IP} both $\Delta(G)$ and $\Gamma(G)$ are acyclic if and only if $B(G)=C_{4}$. In this case $\Delta(G)\simeq\Gamma(G)\simeq P_2$. On the other hand, if $n\geq 6$, then Lemma ~\ref{rem:01} implies that both $\Delta(G)$ and $\Gamma(G)$ are cycles. So $\Delta(G)$ is either a cycle or a path (which is a tree). Now ~\cite[Theorem C]{H} implies that $\Delta(G)$ has at most four vertices. Thus $B(G)$ can be $C_4$, $C_6$ or $C_8$. We claim that $B(G)$ cannot be a cycle of length eight. Otherwise, if $B(G)=C_8$, then $\Delta(G)\simeq\Gamma(G)=C_4$. First suppose that $G$ is solvable. By the main theorem of ~\cite{LMe}, we have $G\simeq H\times K$, where $\rho(H)=\{p,q\}$, $\rho(K)=\{r,s\}$ and both $\Delta(H)$ and $\Delta(K)$ are disconnected graphs. This implies that there exists $m,n\in cd(H)^{*}$ and $l,k\in cd(K)^{*}$ such that $m=p^{\alpha}$, $n=q^{\beta}$, $l=r^{\gamma}$ and $k=s^{\delta}$, for some positive integers $\alpha$, $\beta$, $\gamma$ and $\delta$. By the structure of $G$ it is clear that $\{1, m, n, l, k, ml, mk, nl, nk\}\subseteq cd(G)$, which contradicts the form of $B(G)$. So in the solvable case, $B(G)$ is not a cycle of length eight. On the other hand, ~\cite{LW} implies that a square (i.e. $C_{4}$) cannot be the prime degree graph of a nonsolvable group. Thus when $G$ is nonsolvable $B(G)$ is not $C_8$.
\end{proof}
\begin{corollary}
Let $G$ be a finite group and $B(G)$ be a cycle. Then $G$ is solvable and $dl(G)\leq |cd(G)|\leq 4$.
\end{corollary}
\begin{proof}
Let $B(G)=C_{n}$. By Theorem ~\ref{thm: 1}, we deduce that $n\in\{4,6\}$ and $\Gamma(G)$ is a complete graph. Since $\Gamma(G)$ is complete, ~\cite{bian} implies that $G$ is solvable. As $|cd(G)|\leq 4$ in the solvable group $G$, we conclude that $dl(G)\leq |cd(G)|\leq 4$ by ~\cite{3}.
\end{proof}
\begin{example}~\label{ex:1}
From ~\cite{LM} we know that for every pair of odd primes $p$ and $q$ such that $p$ is congruent to $1$ modula $3$ and $q$ is a prime divisor of $p+1$, there exists a solvable group $G$ such that $cd(G)=\{1,3q,p^{2}q,3p^{3}\}$. This gives an example of a solvable group $G$ whose bipartite divisor graph related to the set of character degrees, is a cycle of length $6$.
\end{example}

\begin{example}
 There are exactly $66$ groups of order $588$. Among these groups, there are exactly two nonabelian groups whose bipartite divisor graphs are  cycles of length four. These groups have $\{1,6,12\}$ as their irreducible character degrees set.
\end{example}
\begin{proposition}
Let $G$ be a nonabelian finite group, $P\in Syl_{p}(G)$ such that $|P|=p^{2}$, where $p\geq 7$, $p \neq 11$ and $[G:P]=12$. $B(G)$ is a cycle if and only if $cd(G)=\{1,6,12\}$.
\end{proposition}
\begin{proof}
 It is easy to see that $P$ is a normal abelian Sylow $p$-subgroup of $G$, so $p$ does not appear in the vertex set of $B(G)$. Therefore, $B(G)$ is a cycle of length four with $\rho(G)=\{2,3\}$. We claim that $\frac{G}{P}$ is abelian. Suppose $\frac{G}{P}$ is nonabelian. This implies that $\frac{G}{P}$ is one of the groups $A_4$, $D_{12}$, or $T\simeq \mathbb{Z}_3\rtimes\mathbb{Z}_4$. In all the above cases we can see that there exists $\chi$ in $Irr(G)$ such that $\chi(1)$ is either $2$ or $3$. Thus $\chi(1)$ is a vertex of degree one in $B(G)$, which contradicts that $B(G)$ is a cycle. Therefore $\frac{G}{P}$ is abelian. This implies that $G$ is either $P\rtimes \mathbb{Z}_{12}$ or $P\rtimes (\mathbb{Z}_2\times\mathbb{Z}_6)$.
 Now by ~\cite[Lemma 2.3]{p}, we have $cd(G)=\{\beta(1)[G:I_{G}(\lambda)] : \lambda\in Irr(P), \beta\in Irr(\frac{I_{G}(\lambda)}{P})\}$. Since $\frac{G}{P}$ is abelian, we conclude that $\frac{I_{G}(\lambda)}{P}$ is abelian and $cd(G)=\{[G:I_{G}(\lambda)] : \lambda\in Irr(P)\}$. As for each $\lambda\in Irr(P)$, $[G:I_{G}(\lambda)]$ divides $[G:P]=12$ and $B(G)$ is a cycle of length four, we conclude that $cd(G)=\{1,6,12\}$.
\end{proof}
\begin{remark}\label{rem:21}
Let $G$ be a finite group with $B(G)=C_{n}$. First suppose that $n=4$ and $\pi(G)=\rho(G)=\{p,q\}$. Since $p$ divides every nonlinear irreducible character degree of $G$, ~\cite[Corollary 12.2]{5} implies that $G$ has a normal $p$-complement $Q$. As $\pi(G)=\rho(G)=\{p,q\}$, $Q$ is the normal Sylow $q$-subgroup of $G$. Let $P$ be a Sylow $p$-subgroup of $G$. Similarly, we can see that $P$ is normal in $G$. Thus all the Sylow subgroups of $G$ are normal which implies that $G$ is nilpotent. Therefore $G$ is the direct product of its Sylow subgroups which contradicts the structure of $B(G)$. So in this case we always have $\rho(G)\subset \pi(G)$. But this is not always the case if $B(G)=C_{6}$, because as we can see in Example ~\ref{ex:1}, $G$ is a group generated by $P$, $\sigma$ and $\tau$, where $P$ is a Camina $p$-group of nilpotent class three and $\sigma$, $\tau$ are two commuting automorphisms of $P$ with orders $q$ and $3$, respectively. As it is explained in ~\cite{LM}, we have $|G|= p^{7}3q$. So in this case $\pi(G)=\rho(G)$.
\end{remark}
Suppose that $G$ is a finite group. The following theorem shows that if $B(G)$ is a cycle of length four, then $G$ has a normal abelian subgroup which explains the structure of $cd(G)$.
\begin{theorem}\label{thm:98}
Let $G$ be a finite group. Assume that $B(G)$ is a cycle of length $4$. There exists a normal abelian subgroup $N$ of $G$ such that $cd(G)=\{[G:I_{G}(\lambda)] : \lambda\in Irr(N)\}$.
\end{theorem}
\begin{proof}
Suppose that $G$ is a finite group of order $p_1^{a_1}p_2^{a_2}...p_l^{a_l}$. Without loss of generality, we may assume that $p_1=p$, $p_2=q$ and $B(G): p-m-q-n-p$. Thus for each $p_i\in \pi(G)\setminus\{p,q\}$, the Sylow $p_i$-subgroup of $G$ is a normal abelian subgroup of $G$. Let $N=P_3\times ...\times P_l$. Since $B(G)$ is a cycle, we deduce from Remark ~\ref{rem:21} that $G$ is not a $\{p,q\}$-group. Thus $N$ is a nontrivial normal abelian subgroup of $G$. Now $\frac{G}{N}$ is a $\{p,q\}$-group, so its bipartite divisor graph is not a cycle of length four. As $cd(\frac{G}{N})\subseteq cd(G)$, there exists no element of $cd(\frac{G}{N})$ which is a prime power. So for each nonlinear $\chi\in Irr(\frac{G}{N})$, $\chi(1)=p^{\alpha}q^{\beta}$, for some positive integers $\alpha$ and $\beta$. This implies that $\frac{G}{N}$ is the direct product of its Sylow subgroups which are nonabelian. But this contradicts the form of $cd(\frac{G}{N})$. Thus $\frac{G}{N}$ is abelian and $G=N\rtimes H$, where $H$ is the hall $\{p,q\}$-subgroup of $G$. Now by ~\cite[Lemma 2.3]{p}, we have $cd(G)=\{\beta(1)[G:I_{G}(\lambda)] : \lambda\in Irr(N), \beta\in Irr(\frac{I_{G}(\lambda)}{N})\}$. Since $\frac{G}{N}$ is abelian, we conclude that $\frac{I_{G}(\lambda)}{N}$ is abelian and $cd(G)=\{[G:I_{G}(\lambda)] : \lambda\in Irr(N)\}$.
\end{proof}
\begin{remark}
Let $G$ be a finite group of order $p_1^{a_1}p_2^{a_2}...p_l^{a_l}$. Assume that $B(G)$ is a cycle of length six and $\pi(G)\neq\rho(G)$. Let $N$ be as in the proof of Theorem ~\ref{thm:98}. So $N$ is nontrivial. If $\frac{G}{N}$ is abelian, similar to the above proof, we can explain $cd(G)$ with respect to $N$. So suppose that $\frac{G}{N}$ is not abelian. Since $B(G)$ is a cycle, each nonlinear irreducible character of $\frac{G}{N}$ is divisible by exactly two primes.  $B(\frac{G}{N})$ is a cycle if and only if it is a cycle of length six and this will occur if and only if $cd(\frac{G}{N})=cd(G)$. So we may assume that $cd(\frac{G}{N})\subset cd(G)$. This implies that $B(\frac{G}{N})$ is either a path of length two or four. According to the proof of Theorem ~\ref{thm:2}, we can see that $\frac{G}{N}$ has a normal $t$-complement for a prime $t\in\{p,q,r\}$ or it has a normal abelian subgroup whose index is the only nontrivial degree of $\frac{G}{N}$ (and so of $G$).
\end{remark}

\end{document}